\documentclass[12pt]{article}
\usepackage{amsmath}
\usepackage{amssymb}
\usepackage{amsthm}
\usepackage{verbatim}
\usepackage{mathrsfs}
\newcommand{\R}[0]{\mathbb R}

\newcommand{\Ds}[0]{\mathcal D}

\newtheorem{Th}{Theorem}[section]
\newtheorem{Lemma}{Lemma}[section]

\begin{document}

\title{On a Lagrangian formulation of the 1D Green-Naghdi system}
\author{H. Inci}

\maketitle

\begin{abstract}
	In this paper we consider the 1D Green-Naghdi system. This system describes the evolution of water waves over a flat bottom in the shallow water regime in terms of the surface height $h$ and the horizontal velocity $u$. We give a Lagrangian formulation of the 1D Green-Naghdi system on a Sobolev type diffeomorphism group. As an application of this formulation we prove local well-posedness for $(h,u)$ in the Sobolev space $(1+H^s(\R)) \times H^{s+1}(\R),\; s > 1/2$. This improves the local well-posedness range for the 1D Green-Naghdi system.
\end{abstract}

\section{Introduction}\label{section_introduction}

The 1D Green-Naghdi system is given by
\begin{align}
	\nonumber
	&u_t + uu_x + h_x=\frac{1}{3h} \partial_x \left(h^3 (u_{tx}  + uu_{xx}  - u_x^2)\right),\quad t \geq 0,\; x \in \R,\\
	\label{gn}
	&h_t+\partial_x (hu) =0,\quad t \geq 0,\;x \in \R,\\
	\nonumber
	&h(t=0)=h_0,\;u(t=0)=u_0,
\end{align}
where $h(t,x) \in \R$ is the height of the upper free surface of the water wave over the flat bottom and $u(t,x) \in \R$ its horizontal velocity. The system \eqref{gn} describes the evolution of water waves over a flat bottom in the shallow water regime, i.e. when the typical wavelength is much larger than the typical water depth.\\
The 2D version of \eqref{gn} for a variable bottom is derived in a paper of Green and Naghdi \cite{gn}. The name ``Green-Naghdi system'' originates from this paper. But the 1D version \eqref{gn} appears already in a paper by Serre \cite{serre} and later in \cite{gardner}. People refer to \eqref{gn} sometimes also as the ``Serre-Green-Naghdi system''.\\ \\
Let us introduce the operator
\begin{equation}\label{Ah}
	A_h: u \mapsto 3hu- \partial_x (h^3 u_x).
\end{equation}
Then a simple calculation shows that the first equation in \eqref{gn} is equivalent to
\[
	A_h(u_t + uu_x)=-3h h_x-2 \partial_x (h^3 u_x^2).
\]
If $A_h$ is invertible one can write \eqref{gn} in non-local form (see also \cite{lwp}) as
\begin{align}
	\nonumber
	&u_t + uu_x =-A_h^{-1}\left(3h h_x+2 \partial_x (h^3 u_x^2)\right),\quad t \geq 0,\; x \in \R,\\
	\label{gn_nonlocal}
	&h_t+\partial_x (hu) =0,\quad t \geq 0,\;x \in \R,\\
	\nonumber
	&h(t=0)=h_0,\;u(t=0)=u_0.
\end{align}
The Green-Naghdi system in the form \eqref{gn_nonlocal} is the starting point for the Lagrangian formulation. More precisely, we want to express \eqref{gn_nonlocal} in the Lagrangian variable $\varphi$, i.e. in terms of the flow map $\varphi$ of $u$. Recall that the flow map is defined as
\[
	\varphi_t(t,x)=u(t,\varphi(t,x)), \; \varphi(0,x)=x,\; t \geq 0,\;x \in \R.
\]
Note that this defines for each $t \geq 0$ a diffeomorphism $\varphi(t):=\varphi(t,\cdot)$ of $\R$. Using the second equation in \eqref{gn_nonlocal} we easily see that
\[
	\frac{d}{dt} \left(\varphi_x(t) \cdot h(t) \circ \varphi(t) \right)=0.
\]
In other words we can write the second equation in \eqref{gn_nonlocal} in the Lagrangian variable $\varphi$ as
\begin{equation}\label{h}
	h(t)=\left(\frac{h_0}{\varphi_x(t)}\right) \circ \varphi(t)^{-1},\;t \geq 0.
\end{equation}
To write the first equation of \eqref{gn_nonlocal} in the Lagrangian variable $\varphi$ consider
\[
	\frac{d}{dt} \varphi_t(t)=\frac{d}{dt} u(t) \circ \varphi(t) = \left(u_t(t)+u(t) u_x(t)\right) \circ \varphi(t). 
\]
If we now replace $u_t+u u_x$ by the corresponding expression from \eqref{gn_nonlocal} and use theirin $u(t)=\varphi_t(t) \circ \varphi(t)^{-1}$ and from \eqref{h} the identity $h(t)=(u_0/\varphi_x(t)) \circ \varphi(t)^{-1}$ we end up with a second order initial value problem
\begin{equation}\label{ode}
	\varphi_{tt}=F(\varphi,\varphi_t,h_0),\; t \geq 0,\;\varphi(0)=\text{id},\;\varphi_t(0)=u_0,
\end{equation}
where $\text{id}:\R \to \R,\;x \mapsto x$, is the identity map on $\R$. To put \eqref{ode} into a proper mathematical framework we need the right functional space for $\varphi$ and the smoothness of $F(\varphi,\varphi_t,h_0)$.\\
The functional space for the Green-Naghdi system \eqref{gn_nonlocal} we have in mind in this paper are the Sobolev spaces. Recall that for $s \geq 0$ the Sobolev space $H^s(\R)$ is defined as
\[
	H^s(\R)=\{f \in L^2(\R) \;|\; \|f\|_{H^s} < \infty \},
\]
where the norm $\|\cdot\|_{H^s}$ is given by
\[
	\|f\|_{H^s}=\left(\int_\R (1+|\xi|^2)^s |\hat f(\xi)|^2 \;d\xi\right)^{1/2}.
\]
Here we denote by $\hat f$ the Fourier transform of $f$. Sobolev spaces of negative order are defined as dual spaces
\[
	H^{-s}(\R)=\left(H^s(\R)\right)',\; s \geq 0.
\]
Suppose now $s > 1/2$. The height $h$ of the free upper surface is an $H^s$ perturbation of the equilibrium height $\bar h \equiv 1$. We take as state space for $h$
\begin{equation}\label{Us}
	U^s=\{h:\R \to \R \;|\; h-1 \in H^s(\R), h(x) > 0 \text{ for all } x \in \R\}.
\end{equation}
By the Sobolev Imbedding Theorem we know that $H^s(\R)$ can be embedded into $C_0(\R)$, the space of continuous functions on $\R$ vanishing at infinity. Thus for $h \in U^s$ we conclude $\inf_{x \in \R} h(x) > 0$, which means that $U^s-1$ is an open subset of $H^s(\R)$. So $U^s$ has naturally a differential structure.  Thus to speak about 
\[
	(h,u) \in C([0,T];U^s \times H^{s+1}(\R)),\; T > 0,
\]
makes sense. In \cite{composition} the authors studied for $s'=s+1 > 3/2$ the functional space
\[
	\Ds^{s'}(\R)=\{\varphi:\R \to \R \;|\; \varphi-\text{id} \in H^{s'}(\R),\; \varphi_x(x) > 0 \text{ for all } x \in \R \}.
\]
By the Sobolev Imbedding $H^{s'}(\R) \hookrightarrow C^1_0(\R)$ one gets that $\Ds^{s'}(\R)$ consists of $C^1$ diffeomorphisms of $\R$ and that $\Ds^{s'}(\R)-\text{id}$ is an open subset of $H^{s'}(\R)$. So $\Ds^{s'}(\R)$ has naturally a differential structure. Moreover, as was shown in \cite{composition}, the maps
\[
	H^\sigma(\R) \times \Ds^{s'}(\R) \to H^\sigma(\R),\;(f,\varphi) \mapsto f \circ \varphi,\; 0 \leq \sigma \leq s',
\]
and
\[
	\Ds^{s'}(\R) \to \Ds^{s'}(\R),\; \varphi \mapsto \varphi^{-1},
\]
are continuous. In particular $\Ds^{s'}(\R)$ is a topological group when the group operation is composition of maps. Now suppose that $(h,u)$ is a solution to the Green-Naghdi system \eqref{gn_nonlocal} on $[0,T]$ with
\[
	(h,u) \in C([0,T];U^s \times H^{s+1}(\R)).
\]
In \cite{lagrangian} it was shown that for $s'=s+1$ there is a unique 
\[
	\varphi \in C^1([0,T];\Ds^{s'}(\R))
\]
satisfying $\varphi_t(t)=u(t) \circ \varphi(t),\;0 \leq t \leq T,\;\varphi(0)=\text{id}$. Thus $\Ds^{s'}(\R)$ is the right functional space for the Lagrangian variable $\varphi$. The first main result of this paper reads then as

\begin{Th}\label{th_analytic}
Let $s > 1/2$. Then
	\[
		\Ds^{s+1}(\R) \times H^{s+1}(\R) \times U^s \to H^{s+1}(\R),\;(\varphi,v,h_0) \mapsto F(\varphi,v,h_0)
	\]
is real analytic. Here $F$ is the map from \eqref{ode}.
\end{Th}

For the basics of analyticity in Banach spaces we refer to \cite{lagrangian}. Using Theorem \ref{th_analytic} and the Picard-Lindel\"of Theorem we get for every $h_0 \in U^s$ and $u_0 \in H^{s+1}(\R)$ a unique local solution $\varphi$ to \eqref{ode} on some time interval $[0,T]$. By defining now
\[
	h(t)=\left(\frac{h_0}{\varphi_x(t)}\right) \circ \varphi(t)^{-1},\;u(t)=\varphi_t(t) \circ \varphi(t)^{-1},\; 0 \leq t \leq T,
\]
we get a solution $(h,u) \in C([0,T];U^s \times H^{s+1}(\R))$ to \eqref{gn_nonlocal}. With this the second main result of the paper reads as

\begin{Th}\label{th_lwp}
Let $s > 1/2$. Then the 1D Green-Naghdi system \eqref{gn_nonlocal} is locally well-posed for $(h,u)$ in $U^s \times H^{s+1}(\R)$.
\end{Th}

In \cite{lwp} it was shown that \eqref{gn_nonlocal} is locally well-posed in $U^s \times H^{s+1}(\R),\; s > 3/2$. Theorem \ref{th_lwp} improves this result.

\section{The operator $A_h$}\label{section_Ah}

The goal of this section is to prove that for $s > 1/2$ and $h \in U^s$ as in \eqref{Us} the operator $A_h$ in \eqref{Ah} is an isomorphism $A_h:H^{s+1}(\R) \to H^{s-1}(\R),\;u \mapsto 3h u - \partial_x(h^3 u_x)$. To do that consider the following inner product on $H^1(\R)$
\[
	\langle u,v \rangle_h = \int_\R 3h u v + h^3 u_x v_x \;dx.
\]
Since $\|h\|_{L^\infty} < \infty$ and $\inf_{x \in \R} h(x) > 0$ we easily see that $\langle \cdot,\cdot \rangle_h$ is equivalent to the $H^1$ inner product
\[
	\langle u,v \rangle_{H^1}=\int_\R u v + u_x v_x \;dx. 
\]

\begin{Lemma}\label{lemma_s1}
Let $s > 1/2$ and $h \in U^s$. Then
	\[
		A_h:H^{s+1}(\R) \to H^{s-1}(\R),\;u \mapsto 3h u - \partial_x(h^3 u_x),
	\]
is an isomorphism.
\end{Lemma}

\begin{proof}
In the following we will use $\langle \cdot,\cdot \rangle$ for the duality pairing between $H^{\sigma}(\R)$ and $H^{-\sigma}(\R)$. Suppose first $1/2 < s \leq 2$ and let $f \in H^{s-1}(\R)$. We have $f \in H^{-1}(\R) =\left(H^1(\R)\right)'$. By the Riesz Representation Theorem there is a unique $u \in H^1(\R)$ s.t.
	\[
		\langle u,\phi \rangle_h=\langle f,\phi \rangle
	\]
for all test functions $\phi \in C_c^\infty(\R)$. We can write this in $H^{-1}(\R)$ as
	\[
		3hu-\partial_x(h^3 u_x) = f.
	\]
	Thus 
	\[
		\partial_x(h^3 u_x)=3hu-f. 
\]
	Since we have by assumption $s-1 \leq 1$ the right hand side is in $H^{s-1}(\R)$. As $h^3 u_x \in L^2(\R)$ and $\partial_x (h^3 u_x) \in H^{s-1}(\R)$ we conclude $h^3 u_x \in H^s(\R)$. From \cite{composition} we know that dividing by $h^3$ is a bounded linear map $H^s(\R) \to H^s(\R)$. We therefore have $u_x \in H^s(\R)$ and with that $u \in H^{s+1}(\R)$ satisfying $A_h(u)=f$. So we've proved that for $1/2 < s \leq 2$
\[
	A_h:H^{s+1}(\R) \to H^{s-1}(\R)
\]
	is an isomorphism. Suppose now $2 < s \leq 3$ and $f \in H^{s-1}(\R)$. The previous step shows $u \in H^3(\R)$ and hence $3hu-f \in H^{s-1}(\R)$. Arguing as before we conclude $u \in H^{s+1}(\R)$. Continuing like that for $3 < s \leq 4,\;4 < s \leq 5,\ldots$ shows that
\[
	A_h:H^{s+1}(\R) \to H^{s-1}(\R),\;s > 1/2,
\]
is an isomorphism.
\end{proof}

\section{Lagrangian formulation}\label{section_lagrangian}

The goal of this section is to prove Theorem \ref{th_analytic}. Let us start by introducing some notation.\\
Let $s > 1/2$ and $0 \leq \sigma \leq s$. Then multiplication
\[
	H^s(\R) \times H^\sigma(\R) \to H^\sigma(\R),\;(f,g) \mapsto f \cdot g,
\]
is continuous -- see \cite{composition}. For $1/2 < s < 1$ multiplication extends for $s-1 \leq \sigma < 0$ to a continuous bilinear map
\[
	H^s(\R) \times H^\sigma(\R) \to H^\sigma(\R),\;(f,g) \mapsto f \cdot g.
\]
This follows from the fact that there is a constant $C > 0$ s.t.
\[
	\left|\int_\R f \cdot g \cdot \phi \;dx \right| \leq \|g\|_\sigma \|f \cdot \phi\|_{-\sigma} \leq C \|f\|_s \|g\|_\sigma \|\phi\|_{-\sigma},
\]
for all $\phi \in C_c^\infty(\R)$, where we used $-\sigma < s$. In particular we have for $\varphi \in \Ds^{s+1}(\R)$ a well-defined multiplication operator 
\[
	M_{\varphi_x}:H^\sigma(\R) \to H^\sigma(\R),\;f \mapsto \varphi_x \cdot f, 
\]
for $\min\{0,s-1\} \leq \sigma \leq s$.
Moreover
\[
	\Ds^{s+1}(\R) \to L(H^\sigma(\R);H^\sigma(\R)),\;\varphi \mapsto M_{\varphi_x},
\]
is affine linear and hence it is analytic. Here we denote by $L(X;Y)$ the space of bounded linear maps from $X$ to $Y$. From \cite{composition} we know that dividing by $\varphi_x$ is a bounded linear map $H^\sigma(\R) \to H^\sigma(\R)$. In other words $M_{\varphi_x}^{-1} \in L(H^\sigma(\R);H^\sigma(\R))$. Using Neumann series we see that inversion of linear maps is an analytic process, hence for $\min\{0,s-1\} \leq \sigma \leq s$ the map
\[
	\Ds^{s+1}(\R) \to L(H^\sigma(\R);H^\sigma(\R)),\;\varphi \mapsto M_{\varphi_x}^{-1}
\]
is analytic. As an immediate consequence we get that the map
\[
	\Ds^{s+1}(\R) \times U^s \to U^s,\; (\varphi,h_0) \mapsto \frac{h_0}{\varphi_x}=M_{\varphi_x}^{-1} h_0,
\]
is analytic.
\\
Let $s > 1/2$ and $\varphi \in \Ds^{s+1}(\R)$. We denote by $R_\varphi:f \mapsto f \circ \varphi$ composition with $\varphi$ from the right. Note that $R_\varphi^{-1}=R_{\varphi^{-1}}$. As mentioned in Section \ref{section_introduction} we know from \cite{composition} that for $0 \leq \sigma \leq s+1$
\[
	R_\varphi:H^{\sigma}(\R) \to H^{\sigma}(\R),\;f \mapsto f \circ \varphi,
\]
is a continuous linear map. If $1/2 < s < 1$ then this extends for $s-1 \leq \sigma < 0$ to a continuous linear map
\[
	R_\varphi:H^\sigma(\R) \to H^\sigma(\R).
\]
The reason is that there is a constant $C > 0$ such that we have for all test functions $\phi \in C_c^\infty(\R)$
\[
	\left|\int_\R f \circ \varphi \cdot \phi \;dx\right| = \left|\int_\R f \cdot \frac{\phi \circ \varphi^{-1}}{\varphi_x \circ \varphi^{-1}} \;dx\right| \leq C \|f\|_\sigma \|\phi\|_{-\sigma}.
\]
This follows from $-\sigma < s$ and the fact that division by $\varphi_x$ and $R_\varphi^{-1}$ are bounded linear maps $H^{-\sigma}(\R) \to H^{-\sigma}(\R)$.\\
The composition map has poor regularity. It is not more than continuous. The reason is that to take the derivative with respect to $\varphi$ in $\varphi \mapsto f \circ \varphi$ we have to take the derivative of $f$, which leads to a loss of derivative. But the conjugation with $R_\varphi^{-1}$ turns out to be smooth.

\begin{Lemma}\label{lemma_conjugation}
Let $s > 1/2$. Then
	\[
		\Ds^{s+1}(\R) \to L(H^{s+1}(\R);H^s(\R)),\;\varphi \mapsto R_\varphi \partial_x R_\varphi^{-1}
	\]
and
	\[
		\Ds^{s+1}(\R) \to L(H^s(\R);H^{s-1}(\R)),\;\varphi \mapsto R_\varphi \partial_x R_\varphi^{-1}
	\]
are analytic.
\end{Lemma}

\begin{proof}
Using the chain rule we have
	\[
		\left(\partial_x(f \circ \varphi^{-1})\right) \circ \varphi = \frac{\partial_x f}{\varphi_x}.
	\]
Thus $R_\varphi \partial_x R_\varphi^{-1}=M_{\varphi_x}^{-1} \partial_x$, which by the above considerations is analytic in $\varphi$.
\end{proof}

We can now prove Theorem \ref{th_analytic}.

\begin{proof}[Proof of Theorem \ref{th_analytic}]
Let $h_0 \in U^s$. We want to show that
	\begin{align*}
		&\Ds^{s+1}(\R) \times H^{s+1}(\R) \to H^{s+1}(\R),\\
		&(\varphi,v) \mapsto F(\varphi,v,h_0)=-R_\varphi A_{(h_0/\varphi_x)\circ \varphi^{-1}}^{-1} \Big(3(h_0/\varphi_x)\circ \varphi^{-1} \cdot \partial_x \left((h_0/\varphi_x)\circ \varphi^{-1}\right)\\
		&+2\partial_x \left(\left((h_0/\varphi_x) \circ \varphi^{-1}\right)^3 \cdot (\partial_x (v \circ \varphi^{-1}))^2 \right)\Big),
	\end{align*}
	is analytic. We rewrite $F(\varphi,v,h_0)$ as
\begin{align*}
	F(\varphi,v,h_0)= &-R_\varphi A_{(h_0/\varphi_x)\circ \varphi^{-1}}^{-1} R_\varphi^{-1} \Big(3M_{\varphi_x}^{-1} h_0 \cdot R_\varphi \partial_x R_\varphi^{-1} M_{\varphi_x}^{-1} h_0\\
	&+2R_\varphi \partial_x R_\varphi^{-1} \left((M_{\varphi_x}^{-1} h_0)^3 \cdot (R_\varphi \partial_x R_\varphi^{-1} v)^2 \right)\Big).
\end{align*}
Consider first the operator $R_\varphi A_{(h_0/\varphi_x)\circ \varphi^{-1}}^{-1} R_\varphi^{-1}$. We clearly have
	\[
		R_\varphi A_{(h_0/\varphi_x)\circ \varphi^{-1}}^{-1} R_\varphi^{-1}= \left(R_\varphi A_{(h_0/\varphi_x)\circ \varphi^{-1}}R_\varphi^{-1}\right)^{-1}.
	\]
	We have for $f \in H^{s+1}(\R)$
	\[
		R_\varphi A_{(h_0/\varphi_x)\circ \varphi^{-1}}R_\varphi^{-1}(f)=3M_{\varphi_x}^{-1}h_0 \cdot f - R_\varphi \partial_x R_\varphi^{-1} \left( (M_{\varphi_x}^{-1}h_0)^3 \cdot R_\varphi \partial_x R_\varphi^{-1} f\right).
	\]
Thus the map
	\[
		\Ds^{s+1}(\R) \times U^s \to L(H^{s+1}(\R);H^{s-1}(\R)),\;(\varphi,h_0) \mapsto R_\varphi A_{(h_0/\varphi_x)\circ \varphi^{-1}}R_\varphi^{-1}(\cdot),
	\]
is analytic. Since inversion of linear maps is an analytic process we get that
	\[
		\Ds^{s+1}(\R) \times U^s \to L(H^{s-1}(\R);H^{s+1}(\R)),\;(\varphi,h_0) \mapsto R_\varphi A_{(h_0/\varphi_x)\circ \varphi^{-1}}^{-1} R_\varphi^{-1}(\cdot),
		\]
is analytic. We clearly have that
	\[
		\Ds^{s+1}(\R) \times U^s \to H^{s-1}(\R),\; (\varphi,h_0) \mapsto 3M_{\varphi_x}^{-1} h_0 \cdot R_\varphi \partial_x R_\varphi^{-1} M_{\varphi_x}^{-1} h_0,
	\]
and
	\begin{align*}
		&\Ds^{s+1}(\R) \times H^{s+1}(\R) \times U^s \to H^{s-1}(\R),\\
		&(\varphi,v,h_0) \mapsto 2R_\varphi \partial_x R_\varphi^{-1} \left((M_{\varphi_x}^{-1} h_0)^3 \cdot (R_\varphi \partial_x R_\varphi^{-1} v)^2 \right),
	\end{align*}
	are analytic maps. So composing the latter two maps with $R_\varphi A_{(h_0/\varphi_x)\circ \varphi^{-1}}^{-1} R_\varphi^{-1}(\cdot)$ shows that $F(\varphi,v,h_0)$ depends analytically on $(\varphi,v,h_0)$. This finishes the proof.
\end{proof}
Consequently we get by using Theorem \ref{th_analytic} a Lagrangian formulation of the Green-Naghdi system \eqref{gn_nonlocal} in the form of an analytic second order ODE on $\Ds^{s+1}(\R)$ given by \eqref{ode}.

\section{Local well-posedness of the Green-Naghdi system}\label{section_lwp}

The goal of this section is to prove the local well-posedness result stated in Theorem \ref{th_lwp}. We will prove this in two steps: local existence and uniqueness. But before we do that we prove the following technical lemma.

\begin{Lemma}\label{lemma_technical}
	Let $s > 1/2$ and $T > 0$. Suppose $g \in C^1([0,T];U^s)$ and $\varphi \in C^1([0,T];\Ds^{s+1}(\R))$. Then $g \circ \varphi^{-1} \in C^1([0,T];(1+H^{s-1}(\R)))$ with
	\[
		\frac{d}{dt} g(t) \circ \varphi^{-1}(t) = g_t(t) \circ \varphi(t)^{-1}-\left(\frac{g_x(t) \varphi_t(t)}{\varphi_x(t)}\right) \circ \varphi(t)^{-1},\;0 \leq t \leq T.
	\]
\end{Lemma}

\begin{proof}
Take a sequence $(g^{(k)})_{k \geq 1} \subset C^1([0,T];U^{s+1})$ s.t. $g^{(k)} \to g$ in $C^1([0,T];U^s)$ as $k \to \infty$. By the Sobolev imbedding $H^{s+1}(\R) \hookrightarrow C^1(\R)$ we can differentiate $g^{(k)} \circ \varphi^{-1}$ pointwise in $t$
	\[
		\frac{d}{dt} g^{(k)}(t) \circ \varphi(t)^{-1} = g_t^{(k)}(t) \circ \varphi(t)^{-1}-\left(\frac{g_x^{(k)}(t) \varphi_t(t)}{\varphi_x(t)}\right) \circ \varphi(t)^{-1},\;0 \leq t \leq T.
	\]
By the Fundamental lemma of calculus we get pointwise
	\[
		g^{(k)}(t) \circ \varphi(t)^{-1}=g^{(k)}(0) \circ \varphi(0)^{-1} + \int_0^t  g_t^{(k)}(s) \circ \varphi(s)^{-1}-\left(\frac{g_x^{(k)}(s) \varphi_t(s)}{\varphi_x(s)}\right) \circ \varphi(s)^{-1}\;ds.
	\]
But this is an identity in $H^{s-1}(\R)$ as well. Taking $k \to \infty$ shows the claim.
\end{proof}

Let us prove now the local existence of solutions to the Green-Naghdi system \eqref{gn_nonlocal}.

\begin{Lemma}\label{lemma_existence}
	Let $s > 1/2$ and $(h_0,u_0) \in U^s \times H^{s+1}(\R)$. Then there is $T > 0$ and 
\[
	(h,u) \in C([0,T];U^s \times H^{s+1}(\R)) \cap C^1([0,T];(1+H^{s-1}(\R)) \times H^s(\R))
\]
	solving \eqref{gn_nonlocal}. Moreover the dependence of $(h,u)$ on $(h_0,u_0)$ is continuous.
\end{Lemma}

We can take a uniform $T > 0$ in a neighborhood of $(h_0,u_0)$. Continuous dependence on $(h_0,u_0)$ means continuity in such a neighborhood with the same $T > 0$.

\begin{proof}[Proof of Lemma \ref{lemma_existence}]
	Using Theorem \ref{th_analytic} and the Picard-Lindel\"of Theorem we get a solution $\varphi \in C^\infty([0,T];\Ds^{s+1}(\R))$ to
	\[
		\varphi_{tt}=F(\varphi,\varphi_t,h_0),\;\varphi(0)=\text{id},\;\varphi_t(0)=u_0,
	\]
on some time interval $[0,T]$ for some $T > 0$. For initial data in a neighborhood of $(h_0,u_0)$ we can take the same $T$. We define
	\[
		h(t):=\left(\frac{h_0}{\varphi_x(t)}\right) \circ \varphi(t)^{-1},\;u(t):=\varphi_t(t) \circ \varphi(t)^{-1},\; 0 \leq t \leq T.
	\]
So by the continuity properties of the composition map we see that
	\[
		(h,u) \in C([0,T];U^s \times H^{s+1}(\R)).
	\]
By the Sobolev imbedding $H^{s+1}(\R) \hookrightarrow C^1(\R)$ we know that $u \in C^1([0,T] \times \R)$. Taking pointwise the $t$ derivative in $u \circ \varphi$ gives
	\[
		\varphi_{tt}=\frac{d}{dt}u \circ \varphi = (u_t+uu_x)\circ \varphi = F(\varphi,\varphi_t,u_0). 
	\]
Entangling the last equality leads to the pointwise identity
	\[
		u_t + uu_x = -A_h^{-1}(3h h_x+2 \partial_x (h^3 u_x)).
	\]
	But this is an identity in $H^s$ as well since $U^s \to L(H^{s-1}(\R);H^{s+1}(\R)),\;h \mapsto A_h^{-1}$ is continuous. Thus we have
	\[
		u \in C^1([0,T];H^s(\R)).
	\]
	and the first equation in \eqref{gn_nonlocal} is satisfied. Using Lemma \ref{lemma_technical} one gets that
\[
	h=\left(\frac{h_0}{\varphi_x}\right) \circ \varphi^{-1} \in C^1([0,T];1+H^{s-1}(\R))
\]
	and the second equation in \eqref{gn_nonlocal} is satisfied. Continuous dependence on the initial data follows from the continuity properties of the composition map. This finishes the proof.
\end{proof}

Now we show uniqueness of solutions to \eqref{gn_nonlocal}.

\begin{Lemma}\label{lemma_uniqueness}
Let $s > 1/2$ and $(h_0,u_0) \in U^s \times H^{s+1}(\R)$. Suppose that
	\[
		(h,u),(\tilde h,\tilde u) \in C([0,T];U^s \times H^{s+1}(\R)) \cap C^1([0,T];(1+H^{s-1}(\R)) \times H^s(\R))
	\]
	are solutions to \eqref{gn_nonlocal} on $[0,T]$ for some $T > 0$. Then $(h,u) \equiv (\tilde h,\tilde u)$ on $[0,T]$.
\end{Lemma}

\begin{proof}
From \cite{lagrangian} there are $\varphi,\tilde \varphi \in C^1([0,T];\Ds^{s+1}(\R))$ satisfying
	\[
		\varphi_t=u \circ \varphi,\tilde \varphi_t=\tilde u \circ \tilde \varphi,\; 0 \leq t \leq T,\; \varphi(0)=\tilde \varphi(0)=\text{id}.
	\]
Taking the pointwise $t$ derivative in $u \circ \varphi$ gives
	\[
		\varphi_{tt}=(u + uu_x) \circ \varphi=-R_\varphi A_h^{-1}(3hh_x+2 \partial_x(h^3 u_x)),
	\]
where in the last equality we used the first equation in \eqref{gn_nonlocal}. But this is an identity in $H^{s+1}(\R)$ as well. So $\varphi$ solves the ODE \eqref{ode} on $[0,T]$. A similar argument shows that $\tilde \varphi$ solves the same initial value problem on $[0,T]$. Thus by uniqueness of solutions to ODEs we get $\varphi \equiv \tilde \varphi$ on $[0,T]$, which implies $(h,u) \equiv (\tilde h,\tilde u)$ on $[0,T]$. This finishes the proof.
\end{proof}

By combinining Lemma \ref{lemma_existence} and Lemma \ref{lemma_uniqueness} we can prove Theorem \ref{th_lwp}.

\begin{proof}[Proof of Theorem \ref{th_lwp}]
The proof follows from Lemma \ref{lemma_existence} and Lemma \ref{lemma_uniqueness}.
\end{proof}

\bibliographystyle{plain}

\flushleft
\author{ Hasan \.{I}nci\\
Department of Mathematics, Ko\c{c} University\\
Rumelifeneri Yolu\\
34450 Sar{\i}yer \.{I}stanbul T\"urkiye\\
        {\it email: } {hinci@ku.edu.tr}
}

\end{document}